\documentclass[a4paper,12pt]{article}

\usepackage[utf8x]{inputenc}
\usepackage{amsmath}
\usepackage{amsfonts}
\usepackage{amssymb}
\usepackage{graphicx}
\usepackage{amsthm}
\usepackage{enumerate}
\usepackage{xcolor}

\newtheorem{theo}{Theorem}
\newtheorem{claim}{Claim}

\newtheorem{prop}{Proposition}

\def\x{{\sf x}}
\def\y{{\sf y}}
\def\z{{\sf z}}
\def\C{{\sf C}}
\def\ux{u^\C_\x}
\def\uxp{u^{\C'}_\x}
\def\uy{u^\C_\y}
\def\uz{u^\C_\z}
\def\vx{v^\C_\x}
\def\vxp{v^{\C'}_\x}
\def\vy{v^\C_\y}
\def\vz{v^\C_\z}
\def\Ec{E^\C}
\def\Uc{U^\C}

\title{Decomposition into two trees with orientation constraints}

\author{Olivier Durand de Gevigney
\thanks{%\texttt{olivier.durand-de-gevigney@g-scop.inpg.fr}
Laboratoire G-SCOP, CNRS, Grenoble-INP, UJF, France}
\thanks{partially supported by the TEOMATRO grant ANR-10-BLAN 0207}}

\begin{document}

\maketitle

\begin{abstract}
  We prove that deciding whether the edge set of a graph can be partitionned
  into two spanning trees with orientation constraints is NP-complete. If $P
  \neq NP$, this disproves a conjecture of Recski \cite{Recski2011}.
\end{abstract}

Let $G=(V,E)$ be a graph. \emph{Orienting} an edge $uv \in E$ means replacing 
the edge $uv$ by an arc $uv$ or an arc $vu$. For $m \in \mathbb{Z}_+^V$, an
\emph{$m$-orientation} of $F \subseteq E$ is an orientation of the edges in $F$
such that the number of arcs of $F$ leaving $v$ is $m(v)$ for each $v \in V$.
Given $b,r \in \mathbb{Z}_+^V$, a \emph{$(b,r)$-partition} of $E$ is a partition
of $E$ into a blue and a red spanning tree such that the blue tree has a
$b$-orientation and the red tree has an $r$-orientation. Recski
\cite{Recski2011} conjectured that the existence of a $(b,r)$-partition can be
decided in polynomial time.  The following theorem anwsers negatively this
question if $P \neq NP$.
\begin{theo}\label{th:RecskiNPC}
  Let $G=(V,E)$ be a graph and $b,r \in \mathbb{Z}_+^V$. Deciding
  whether there exists a $(b,r)$-partition of $E$ is NP-complete.
\end{theo}

To prove this result we give a reduction of an instance of \textsc{NotAllEqual
$3$-Sat}. In such an instance, each clause consists of three non-negated
variables and an assignment is a coloring of the variables with blue or red. A
clause is satisfied if it contains both a blue and a red variable. Schaefer
\cite{Schaefer1978} proved that this variation of \textsc{Sat} is NP-complete.
\begin{theo}[\cite{Schaefer1978}]\label{th:Schaefer1978}
  \textsc{NotAllEqual $3$-Sat} is NP-complete.
\end{theo}

Let $\Pi$ be an instance of \textsc{NotAllEqual $3$-Sat} and denote $n$ the
number of clauses. We will define a graph $G(\Pi)=(V,E)$ on $12n+1$ vertices
and two outdegree vectors $b,r \in \mathbb{Z}_+^V$ with the following property.
\begin{claim}\label{cl:Equivalence}
  There exists a $(b,r)$-partition of $E$ if and only if there exists a
  coloring of the variables satisfying $\Pi$.
\end{claim}
\noindent Hence Theorem \ref{th:RecskiNPC} will follow from Theorem
\ref{th:Schaefer1978} and Claim \ref{cl:Equivalence}.\\

For each clause $\C$ of $\Pi$ we add a copy $\C'$ of that clause. Hereinafter 
$\C$ will always denote a clause that is originaly in $\Pi$ and $\C'$ will
always denote a copy.

For each clause $\C=(\x,\y,\z)$ we construct a $\C$-gadget on six vertices
$u^\C_\x, u^\C_\y, u^\C_\z, v^\C_\x, v^\C_\y, v^\C_\z$. This gadget consists of
the triangle on the vertex set $\Uc=\{u^\C_\x,u^\C_\y,u^\C_\z\}$ and the edge
set $\Ec=\{u^\C_\x v^\C_\x, u^\C_\y v^\C_\y, u^\C_\z v^\C_\z\}$ (see Figure
\ref{fg:clauseGadget}).  This construction is also done for the copy $\C'$ of
$\C$.  The out-degree vectors $b$ and $r$ are defined by $b=r=1$ for each of the
$6$ vertices denoted by the letter $u$ in the $\C$-gadget and the $\C'$-gadget,
$b=1$ and $r=0$ for each of the $3$ vertices denoted by the letter $v$ in the
$\C$-gadget, and $b=0$ and $r=1$ for each of the $3$ vertices denoted by the
letter $v$ in the $\C'$-gadget.
\begin{figure}[ht]
  \centering
  \includegraphics{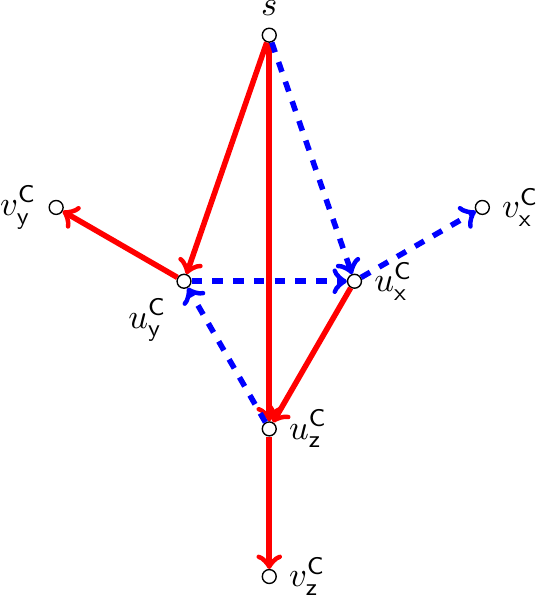}
  \caption{A clause gadget for $\C=(\x,\y,\z)$. The coloring (dashed is blue and
  plain is red) and the orientation of the edges corresponds to a blue coloring
  of $\x$ and a red coloring of $\y$ and $\z$.} 
  \label{fg:clauseGadget}
\end{figure}

We add a special vertex $s$ and, for each vertex $u$ of the $6n$ vertices
denoted by the letter $u$ we add the edge $su$. The out-degree vectors are
defined on $s$ by $b=r=3n$.

For each variable $\x$, we add a cycle $\Delta_\x$ on the vertices of type
$v^\C_\x$ and $v^{\C'}_\x$ where $\C$ contains $\x.$ This cycle alternates
vertices from $\C$-gadgets and vertices from $\C'$-gadgets (see Figure
\ref{fg:variableGadget}). This ends the definition of $G(\Pi)$.
\begin{figure}[ht]
  \centering
  \includegraphics{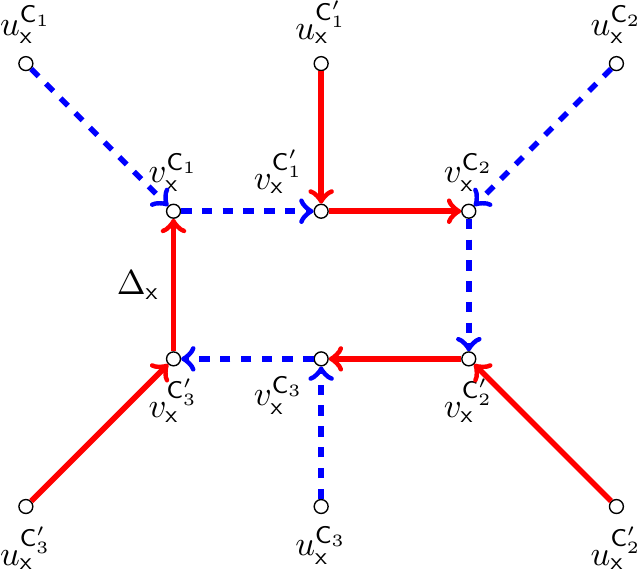}
  \caption{A variable gadget where the original clauses $\C_1$, $\C_2$ and
  $\C_3$ contain the variable $\x$. The coloring (dashed is blue and plain is
  red) of the edges corresponds to a blue coloring of $\x.$} 
  \label{fg:variableGadget}
\end{figure}

\begin{prop}\label{prop:clauseGadget}
  In every $b$-orientation of the blue tree and $r$-orientation of the red tree
  of a $(b,r)$-partition of $E$, for every clause $\C=\{\x,\y,\z\}$, each arc
  from $\Ec$ leaves $\Uc$, $\Ec$ contains both a blue and a red edge and both
  trees restricted to $s \cup \Uc$ are connected. This holds also for copies
  $\C'$ of the original clauses.
\end{prop}
\begin{proof}
  Observe that the neigbors of $s$ are the vertices denoted by the letter $u$ in the
  $\C$-gadgets and the $\C'$-gadgets and there are $3$ neigbors in each of those
  $2n$ gadgets. So we have $d_G(s)=6n=b(s)+r(s)$ hence
  \begin{equation}\label{eq:edgesLeaveS}\tag{$\star$}
    \textrm{all the arcs incident to } s \textrm{ leave } s. 
  \end{equation}
  Hence, by $r(\Uc)=b(\Uc)=3$, each of the $6$ arcs incident to $\Uc$ in $G-s$
  leaves a vertex of $\Uc$, exactly $3$ are blue and exactly $3$ are red. So the
  arcs from $\Ec$ leave $\Uc$. The set $\Ec$ contains both a blue and a red edge
  otherwise one of the tree would contain the triangle $\ux \uy \uz$.  

  Hence, by permuting $\x$, $\y$ and $\z$ if necessary, we may assume that the
  edge $\ux \vx$ is blue and $\uy \vy$ and $\uz \vz$ are red. Thus the triangle
  $\ux \uy \uz$ contains exactly two blue edges and, by \eqref{eq:edgesLeaveS}
  and $r(\ux)=1$, the common end vertex of those two blue edges is not $\ux$. By
  permuting $\y$ and $\z$ if necessary, we may assume that $\ux \uy$ and $\uy
  \uz$ are blue and $\ux \uz$ is red.
  One of the edges $s \ux$, $s \uy$ is blue, otherwise the red tree would
  contain the triangle $s \ux \uz$, and there is at most one blue edge from $s$
  to $\Uc$, otherwise the blue tree would contain one of the cycles $s \ux \uy
  s$, $s \uy \uz s$, $s \ux \uy \uz s$. So either $s\ux$ is blue and $s\uy,s\uz$
  are red or $s\uz$ is blue and $s\ux, s\uy$ are red. In both cases each of the
  tree restricted to $s \cup \Uc$ is connected.
\end{proof}
\begin{prop}\label{prop:variableGadget}
  Let $\x$ be a variable. In every $(b,r)$-partition of $E$, all the edges $\ux
  \vx$, where $\C$ is an original clause containing $\x$, have the same color
  and all the edges $\uxp \vxp$, where $\C'$ is a copy of an original clause
  containing $\x$, have the other color.
\end{prop}
\begin{proof}
  By Proposition \ref{prop:clauseGadget}, in a $b$-orientation of the blue tree
  and an $r$-orientation of the red tree, the arcs of type $\ux \vx$ or $\uxp
  \vxp$ enter the cycle $\Delta_\x$. Hence, since in $\Delta_\x$ vertices 
  with $r=1$ and $b=0$ and vertices with $r=0$ and $b=1$ alternate, $\Delta_x$
  has a circuit orientation and the color of the edges alternates. 

  Let $\C_1$ be an original clause containing $\x$ and suppose that $u^{\C_1}_\x
  v^{\C_1}_\x$ is blue. Denote $v^{\C_2'}_\x$ the neighbor of $v^{\C_1}_\x$ in
  $\Delta_\x$ such that $v^{\C_1}_\x v^{\C_2'}_\x$ is blue. By Proposition
  \ref{prop:clauseGadget}, there exist a blue path joining $s$ and $u^{\C_1}_\x$
  in $s \cup U^{\C_1}$ and a blue path joining $s$ and $u^{{\C'}_2}_\x$ in $s
  \cup U^{\C'_2}$. Thus the edge $u^{\C'_2}_\x v^{\C'_2}_\x$ is red otherwise
  the blue tree would contain a cycle including those two paths and the path
  $u^{\C_1}_\x v^{\C_1}_\x v^{\C'_2}_\x u^{\C'_2}_\x.$ The same argument shows
  that the edge $u^{\C_3}_\x v^{\C_3}_\x$ is blue where $v^{\C_3}_\x$ is the other
  neighbor of $v^{\C'_2}_\x$ in $\Delta_\x.$ Hence a repeated application of this
  argument proves the proposition.
\end{proof}

\begin{proof}[Proof of Claim \ref{cl:Equivalence}]
  Suppose there exists a $(b,r)$-partition of $E$. By Proposition
  \ref{prop:variableGadget}, for each variable $\x$, all the edges of type $\ux
  \vx$, where $\C$ is an original clause containing $\x$, have the same color.
  Hence it is consistent to color a variable $\x$ with the color of an edge $\ux
  \vx$ for an original clause $\C$ containing $\x$. By Proposition
  \ref{prop:clauseGadget}, for each original clause $\C$ of $\Pi$, $\Ec$
  contains a blue and a red edge, thus $\C$ contains a blue and a red variable,
  that is, $\C$ is satisfied. It follows that this coloring satisfies $\Pi$.

  Now suppose that there exists a coloring satisfying $\Pi$. For each variable
  $\x$ and each original clause $\C$ containing $\x$, color $\ux \vx$ with the
  color of $\x$ and color $\uxp \vxp$ with the other color. Since each clause
  $\C$ contains a blue and a red variable, the coloring of the edges induced by
  $s \cup \Uc$ can be done as in Figure \ref{fg:variableGadget} (permute the
  variables and the colors if necessary). Do the same for the coloring of the
  edges in the $\C'$-gadgets. For each variable $\x$ alternate the color along
  the cycle $\Delta_\x$. So far we obtained a partition of $E$ into a blue and a
  red spanning tree. 

  Now orient the edges incident to every $\C$-gadget or $\C'$-gadget as in
  Figure \ref{fg:variableGadget} (or the inverse coloring of that figure).
  Observe that the multiplicity of colors is the same in $\Ec$ and in the edges
  from $s$ to $\Uc$. Hence, for each clause $\C$, there are exactly $3$ blue
  edges and $3$ red edges from $s$ to $\Uc \cup U^{\C'}$ and the outdegree of
  $s$ is $\frac{1}{2}d_G(s)=3n$ in each tree. For each variable $\x$ orient the
  bicolored cycle $\Delta_\x$ to obtain a circuit that satisfies the outdegree
  contraints on vertices of type $\vx$ and $\vxp.$ Hence we obtain a
  $b$-orientation of the blue tree and an $r$-orientation of the red tree.
\end{proof}

\bibliographystyle{plain}
\bibliography{/data/travail/bibliotheque/bibtexfile/biblio}

\end{document}